\documentclass[12pt,reqno]{amsart}
\usepackage{amsmath, amsthm, amssymb, stmaryrd}
\usepackage{hyperref}
\usepackage{enumerate}
\usepackage{enumitem}
\usepackage{url}
\usepackage{tikz-cd}

\let\OLDthebibliography\thebibliography
\renewcommand\thebibliography[1]{
  \OLDthebibliography{#1}
  \setlength{\parskip}{0pt}
  \setlength{\itemsep}{0pt plus 0.3ex}
}

\usepackage{mathtools}

\usepackage{multirow, array}
\usepackage{placeins}
\usepackage{caption} 

\newtheorem{thm}{Theorem}[section]
\newtheorem{lemma}[thm]{Lemma}

\newtheorem{cor}[thm]{Corollary}

\theoremstyle{definition}
\newtheorem{defn}[thm]{Definition}

\theoremstyle{remark}

\numberwithin{equation}{section}

\newcommand*\wrapletters[1]{\wr@pletters#1\@nil}
\def\wr@pletters#1#2\@nil{#1\allowbreak\if&#2&\else\wr@pletters#2\@nil\fi}

\def\le{\leqslant} \def\ge{\geqslant}

\def \bA {\mathbf A}

\def \bbF {\mathbb F}

\def \bN {\mathbb N}

\def \bZ {\mathbb Z}

\def \ba {\mathbf a}

\def \bx {\mathbf x}

\def \fl {\mathfrak l}

\def \fp {\mathfrak p}
\def \fq {\mathfrak q}

\def \cL {\mathcal L}

\def \cP {\mathcal P}
\def \cQ {\mathcal Q}

\def \det {\mathrm{det}}

\def \deg {\mathrm{deg}}

\begin{document}
\title[Expansion properties of polynomials over finite fields]{Expansion properties of polynomials over finite fields}
\author[Nuno Arala \and Sam Chow]{Nuno Arala \and Sam Chow}
\address{Mathematics Institute, Zeeman Building, University of Warwick, Coventry CV4 7AL}
\email{Nuno.Arala-Santos@warwick.ac.uk}
\address{Mathematics Institute, Zeeman Building, University of Warwick, Coventry CV4 7AL}
\email{Sam.Chow@warwick.ac.uk}
\thanks{}
\date{}
\begin{abstract} We establish expansion properties for suitably generic polynomials of degree $d$ in $d+1$ variables over finite fields. In particular, we show that if $P\in\bbF_q[x_1,\ldots,x_{d+1}]$ is a polynomial of degree $d$ coming from an explicit, Zariski dense set, and $X_1,\ldots,X_{d+1}\subseteq\bbF_q$ are suitably large, then $|P(X_1,\ldots,X_{d+1})|=q-O(1)$. 
Our methods rely on a higher-degree extension of a result of Vinh on point--line incidences over a finite field. 
\end{abstract}
\maketitle

\section{Introduction}

For a field $\bbF$, a polynomial $P\in\bbF[x_1,\ldots,x_k]$ and sets $A_1,\ldots,A_k\subseteq\bbF$, we define the set
\[
P(A_1,\ldots,A_k) =
\{P(a_1,\ldots,a_k):
a_i \in A_i
\text{ for }i=1,\ldots,k\}.
\]
\emph{Expansion properties} of $P$ are assertions that if $A_1,\ldots,A_k\subseteq\bbF$ are not too large, then the image $P(A_1,\ldots,A_k)$ is substantially larger than each individual $A_i$. 
The \emph{Elekes--R\'onyai} problem \cite{deZeeuw} is to show that expansion occurs unless $P$ has a very specific shape, e.g. $P$ is additively structured or multiplicatively structured.
For $\bbF=\bbF_q$ a finite field, stronger forms of expansion assert that if $A_1,\ldots,A_k$ are suitably large subsets of $\bbF_q$ then $|P(A_1,\ldots,A_k)| \approx q$.

Tao \cite{Tao}, inspired by the slightly more restrictive hierarchy put forward in \cite{HLS}, proposed the following hierarchy for expansion statements.
\begin{enumerate}[label=(\roman*)]
\item\emph{Weak asymmetric expansion}: there exist absolute constants $c, C>0$ such that $$|P(A_1,\ldots,A_k)|\geq C^{-1}
\min\{|A_1|,\ldots,|A_k|\}^{1-c}q^c$$ whenever $|A_1|,\ldots,|A_k|\geq Cq^{1-c}$.
\item\emph{Moderate asymmetric expansion}: there exist absolute constants $c, C>0$ such that $$|P(A_1,\ldots,A_k)|\geq C^{-1}q$$ whenever $|A_1|,\ldots,|A_k|\geq Cq^{1-c}$.
\item\emph{Almost strong asymmetric expansion}: there exist absolute constants $c,C>0$ such that $$|P(A_1,\ldots,A_k)|\geq q-Cq^{1-c}$$ whenever $|A_1|,\ldots,|A_k|\geq Cq^{1-c}$.
\item\emph{Strong asymmetric expansion}: there exist absolute constants $c,C>0$ such that $$|P(A_1,\ldots,A_k)|\geq q-C$$ whenever $|A_1|,\ldots,|A_k|\geq Cq^{1-c}$.
\item\emph{Very strong asymmetric expansion}: there exist absolute constants $c,C>0$ such that $$P(A_1,\ldots,A_k)=\bbF_q$$ whenever $|A_1|,\ldots,|A_k|\geq Cq^{1-c}$.
\end{enumerate}
The qualifier ``asymmetric'' in the above refers to the fact that the sets $A_1,\ldots,A_k$ are allowed to differ. It is worth noting that most of these forms of expansion are trivial for a fixed $q$ and a fixed polynomial $P$ (``fixed'' meaning that the constants are allowed to depend on them). They become non-trivial when $q$ is allowed to vary, with $P$ varying as well along suitable families of polynomials.

In order to state the most general result in this paper, it is convenient to introduce the following technical definition.

\begin{defn}
Let $P\in\bbF[x_1,\ldots,x_n]$ be a polynomial of degree $d$, where $n\geq2$. We say that $P$ is \emph{nice} if the following holds up to a permutation of the variables $x_1,\ldots,x_n$: if one writes
$$P(x_1,\ldots,x_n)=ax_n^d+\sum_{k=1}^dP_k(x_1,\ldots,x_{n-1})x_n^{d-k}\text{,}$$
where each $P_k$ is a polynomial of degree at most $k$, then the polynomials $P_1,\ldots,P_d\in\bbF[x_1,\ldots,x_n]$ are algebraically independent.
\end{defn}

If we have $n=d+1$ in the notation of the definition above --- as will typically be the case in this paper 
--- and the polynomial
$$\det\left(\frac{\partial}{\partial x_j}P_k\right)_{k,j=1,\ldots,d}$$
is not the zero polynomial, then  $P_1,\ldots,P_d$ are algebraically independent; see for example
\cite[Ch. I, (11.4)]{Lefschetz}. The converse fails in positive characteristic, but this observation gives a sufficiency criterion for niceness that a generic polynomial satisfies and that is easy to check in practice. 

\subsection*{Notation}

We use the Vinogradov notations $\ll$ and $\gg$, as well as the Bachmann--Landau notation $O(\cdot)$.
Let $f$ and $g$ be complex-valued functions. We write $f \ll g$ or $f=O(g)$ if $|f|\leq C|g|$ pointwise, for some constant $C>0$. We use a subscript within these notations to indicate possible dependence for the implied constant $C$. 

\subsection*{Statement of results}

We are now ready to state the most general result in this paper.

\begin{thm}
\label{maingen}
Let $d\in\bN$, and let $q$ be a prime power. Then there exists a constant $C>0$ such that, if $P\in\bbF_q[x_1,\ldots,x_{d+1}]$ is a nice polynomial of degree $d$, and $X_1,\ldots,X_{d+1}\subseteq\bbF_q$ satisfy $|X_k|\geq Cq^{d/(d+1)}$ for $k=1,2,\ldots,d+1$, then
$$
|P(X_1,\ldots,X_{d+1})|=q-O_d\left(\frac{q^{d+1}}{|X_1|\cdots|X_{d+1}|}\right)\text{.}
$$
\end{thm}

For fixed $d$, this result sits between almost strong expansion and strong expansion, in that it neither is implied by the former nor implies the latter. For example, the following corollary, establishing almost strong expansion for nice polynomials of degree $d$ in $d+1$ variables, is now immediate.

\begin{cor}
\label{modgen}
Let $d\in\bN$, and let $q$ be a prime power. Let $C,\varepsilon>0$. If $P\in\bbF_q[x_1,\ldots,x_{d+1}]$ is a nice polynomial of degree $d$, and subsets $X_1,\ldots,X_{d+1}\subseteq\bbF_q$ satisfy $|X_k|\geq Cq^{(1+\varepsilon)d/(d+1)}$ for $k=1,2,\ldots,d+1$, then
$$|P(X_1,\ldots,X_{d+1})|\geq q-O_{d,C,\varepsilon}(q^{1-\varepsilon})\text{.}$$
\end{cor}

Similarly, we have the following corollary, which differs from a strong expansion statement in that it requires each of $X_1,\ldots,X_{d+1}$ to contain a positive proportion of the elements of the field.

\begin{cor}
\label{notstronggen}
Let $d\in\bN$, and let $q$ be a prime power. Let $\delta>0$. If $P\in\bbF_q[x_1,\ldots,x_{d+1}]$ is a nice polynomial of degree $d$, and subsets $X_1,\ldots,X_{d+1}\subseteq\bbF_q$ satisfy $|X_k|\geq \delta q$ for $k=1,2,\ldots,d+1$, then
$$|P(X_1,\ldots,X_{d+1})|\geq q-O_{d,\delta}(1)\text{.}$$
\end{cor}

In the special case when $P$ is a ternary quadratic polynomial --- so that $d=2$ in our notation --- the technical hypothesis that $P$ is nice can be given a more satisfactory interpretation. Indeed we will prove, by a simple computation, the following niceness criterion for ternary quadratic polynomials, which matches the main assumption in \cite{PVZ2019}.

\begin{lemma}
\label{nicequad}
Let $q$ be an odd prime power and let $Q\in\bbF_q[x,y,z]$ be a ternary quadratic polynomial. Then $Q$ is nice if and only if $Q(x,y,z)$ depends on each variable and does not equal $g(h(x)+k(y)+\ell(z))$ for some univariate polynomials $g,h,k,\ell$.
\end{lemma}

Therefore, in the case where $d=2$, the niceness assumption in Theorem \ref{maingen} and Corollaries \ref{modgen} and \ref{notstronggen} can be replaced by the assumption above. This strengthens the main result in \cite{PVZ2019}, which only implies moderate expansion.

This assumption that $Q$ does not take the form $g(h(x)+k(y)+\ell(z))$ cannot be removed, and in fact Theorem \ref{maingen} and its corollaries fail for any quadratic polynomial that violates it. We sketch why this is the case assuming that $q = p$ is a sufficiently large prime and $$Q(x,y,z)=ax^2+by^2+cz^2,$$ where $a,b,c\in\bbF_p^\times$, but the general case follows similarly.
For $s\in\bbF_p$, denote by $r(s)$ the unique integer in $[0,p-1]$ that maps to $s$ under the canonical projection $\mathbb{Z}\to\bbF_p$. Define 
\begin{align*}
&X=\{x\in\bbF_p:
1 \leq r(ax^2) \leq p/4\}\text{.}
\\
&Y=\{y\in\bbF_p:
1 \leq r(by^2) \leq p/4\}\text{,}\\
&Z=\{z\in\bbF_p:
1\leq r(cz^2) \leq p/4\}
\text{.}
\end{align*}
It is a standard exercise in finite Fourier analysis to show using the P\'olya--Vinogradov inequality that
$$|X|,|Y|,|Z|=
\left(\frac{1}{4}+o(1)\right)p\text{.}$$
On the other hand, it is clear that, for any $x\in X$, $y\in Y$ and $z\in Z$, one has $1 \leq r(Q(x,y,z)) \leq 3p/4$ and hence $|Q(X,Y,Z)|\leq 3p/4$, which means that the conclusion of Corollary \ref{notstronggen} does not hold for $Q$.

The question of necessity for the hypothesis that $Q$ depends on each variable is much more delicate. Indeed, questions about expansion properties of binary polynomials are notoriously hard. The strongest results up to date in this regard can be found in \cite{Tao}.




\bigskip

We have not been able to establish similar expansion results for general ternary polynomials of higher degree. However, our approach does succeed in some natural cases.

\begin{thm}\label{conc} Let $d \in \bN$, let $q$ be a prime power. Then there exists a constant $C > 0$ such that the following holds. Let $a \in \bbF_q$, let $F, G \in \bbF_q[y,z]$ be algebraically independent, and let
\[
P(x,y,z) = ax^d + F(y,z)x + G(y,z).
\]
Let $X, Y, Z \subseteq \bbF_q$ satisfy
$|X|, |Y|, |Z| \ge Cq^{2/3}$. Then
\[
|P(X,Y,Z)| = q - O_{\deg(F), \deg(G)}
\left(
\frac{q^3}{|X|\cdot |Y| \cdot |Z|} \right)\text{.}
\]
\end{thm}

\bigskip

There are several results in the literature establishing expansion properties of specific families of higher-degree polynomials, further to the aforementioned work of Pham, Vinh, and de Zeeuw \cite{PVZ2019}. To quote some examples, very strong expansion for $x^2+xy+z$ was established by Skhredov \cite{Shkredov}, building upon work of Bourgain~\cite{Bourgain} on expansion for $x^2+xy$. Vinh \cite{Vinh2} proved moderate expansion results for quaternary polynomials of the form $f(x_1) + g(x_2)+x_3x_4$, $f(x_1)+g(x_2)+(x_3-x_4)^2$, $f(x_1)g(x_2)+x_3x_4$ and $f(x_1)g(x_2)+(x_3-x_4)^2$, under very general conditions on the univariate polynomials $f$ and $g$. In studying the Erd\H{o}s distance problem, Iosevich and Rudnev \cite{IR} demonstrated very strong expansion for
\[
\sum_{i \le d} (x_i - y_i)^2
\]
whenever $d \ge 2$.

The study of expanding polynomials and that of incidence geometry often go hand in hand. Our proof makes use of the following incidence geometry result of Vinh \cite{Vinh1}, which was proved using methods from spectral graph theory that can be found in \cite{AlonSpencer}.

\begin{thm} [Vinh]
\label{VinhThm}
Let $\cP$ be a set of points and $\cL$ be a set of lines in $\bbF_q^2$. Then
\[
\left|
\# \{ (\fp, \fl) \in \cP \times \cL: \fp \in \fl \} - 
\frac{|\cP| \cdot |\cL|}{q} 
\right|
\le
q^{1/2} \sqrt{|\cP|\cdot |\cL|}.
\]
\end{thm}

\subsection*{Organisation}

In \S \ref{IncidenceGeometry}, we use Vinh's result in incidence geometry --- Theorem \ref{VinhThm} --- to obtain a higher-degree generalisation. Then, in \S \ref{MainProof}, we use this to prove our main result, Theorem \ref{maingen}. That section also contains a proof of Lemma \ref{nicequad}, classifying nice quadratic polynomials in three variables. Finally, in \S \ref{TernaryHigher}, we prove Theorem \ref{conc}, on expanding ternary polynomials of higher degree.

\subsection*{Funding}

NA was funded through the Engineering
and Physical Sciences Research Council Doctoral Training Partnership at the
University of Warwick.

\subsection*{Acknowledgements}

We thank Akshat Mudgal and George Shakan for helpful discussions.

\subsection*{Rights}

For the purpose of open access, the authors have applied a Creative Commons Attribution (CC-BY) licence to any Author Accepted Manuscript version arising from this submission.

\section{Incidence geometry}
\label{IncidenceGeometry}

In this section, we prove a result on incidences between points and graphs of polynomials of a fixed degree, which will be central to the proof of the main results. This is a higher-degree generalisation of Theorem \ref{VinhThm}, and might be of independent interest.

\begin{thm} \label{higher}
Let $n \in \bN$, and let $q>n$ be a prime power. Let $\cP$ be a set of points in $\bbF_q^2$, and $\cQ$ a collection of subsets of $\bbF_q^2$ of the shape
\[
\{ (x,y) \in \bbF_q^2:
y = a_n x^n + \cdots + a_0 \}\text{.}
\]
Then
\[
\left|
\# \{ (\fp, \fq) \in \cP \times \cQ: \fp \in \fq \} -
\frac{|\cP| \cdot |\cQ|}{q} 
\right|
\le
q^{n/2} \sqrt{|\cP|\cdot |\cQ|}.
\]
\end{thm}

\begin{proof}
For each $\mathbf{a}=(a_n,\ldots,a_0)\in\bbF_q^{n+1}$, let 
\[
\fq_\mathbf{a}=\{ (x,y) \in \bbF_q^2:
y = a_n x^n + \cdots + a_0 \}.
\]
We observe that the condition $q>n$ implies that the set $\fq_\mathbf{a}$ uniquely determines the vector $\mathbf{a}$.

For each $\mathbf{a}\in\bbF_q^{n-1}$, let
$$\cQ_\mathbf{a}=\cQ\cap\{\fq_{\mathbf{a},a_1,a_0}:a_1,a_0\in\bbF_q\}\text{.}$$
For each $\mathbf{a}\in\bbF_q^{n-1}$ let $\psi_\mathbf{a}:\bbF_q^2\to\bbF_q^2$ be defined by
$$\psi_\mathbf{a}(x,y)=(x,y-a_nx^n-\cdots-a_2x^2)\text{.}$$
Theorem \ref{VinhThm}, applied to the set of points $\psi_\mathbf{a}(\cP)$ and the set of lines given by $y=a_1x+a_0$ with $\fq_{\mathbf{a},a_1,a_0}\in\cQ$, implies that for each $\mathbf{a}\in\bbF_q^{n-1}$,
$$\left|\# \{ (\fp, \fq) \in \cP \times \cQ_{\mathbf{a}}: \fp \in \fq \} -
\frac{|\cP| \cdot |\cQ_\mathbf{a}|}{q} 
\right|
\le
q^{1/2} \sqrt{|\cP|\cdot |\cQ_\mathbf{a}|}\text{.}$$
Now
\begin{align*}
&\left|
\# \{ (\fp, \fq) \in \cP \times \cQ: \fp \in \fq \} -
\frac{|\cP| \cdot |\cQ|}{q} 
\right|\\
&\leq\sum_{\mathbf{a}\in\bbF_q^{n-1}}\left|\# \{ (\fp, \fq) \in \cP \times \cQ_{\mathbf{a}}: \fp \in \fq \} -
\frac{|\cP| \cdot |\cQ_\mathbf{a}|}{q} 
\right|\\
&\leq q^{1/2}\sqrt{|\cP|}\sum_{\mathbf{a}\in\bbF_q^{n-1}}\sqrt{|\cQ_{\mathbf{a}}|}\\
&\leq q^{n/2}\sqrt{|\cP|\cdot\sum_{\mathbf{a}\in\bbF_q^{n-1}}|\cQ_\mathbf{a}|}= q^{n/2}\sqrt{|\cP|\cdot|\cQ|}\text{,}
\end{align*}
where the last inequality follows from Cauchy--Schwarz.
\end{proof}

\section{Proof of the main result}
\label{MainProof}

In this section, we establish Theorem \ref{maingen} and Lemma \ref{nicequad}. We will need the following preliminary result.

\begin{lemma}
\label{bounddeg}
Let $D,n \in \bN$. Then there exists a constant $C_{D,n}$ for which the following holds. Let $\bbF$ be a field, and let $\varphi:\bA_{\bbF}^n\to\bA_{\bbF}^n$ be a dominant morphism given by polynomials of degree at most $D$. Then there exists a proper subvariety $W\subseteq \bA_{\bbF}^n$ of degree at most $C_{D,n}$ such that, for any $\mathbf{y}\in \bA_{\bbF}^n\setminus W$, the fibre of $\varphi$ over $\mathbf{y}$ is zero-dimensional.
\end{lemma}

\begin{proof}
Let $\varphi$ be given by
$$\varphi(\mathbf{x})=(P_1(\mathbf{x}),\ldots,P_n(\mathbf{x}))
\text{.}$$
The condition that $\varphi$ is dominant is equivalent to the statement that $P_1,\ldots,P_n$ are algebraically independent.

Let $M > 0$ be a parameter to be chosen later. For each $1\leq k\leq n$, we seek a polynomial $Q_k$ of degree at most $M$ such that 
\begin{equation}
\label{polyrel}
Q_k(x_k,P_1(\mathbf{x}),\ldots,P_n(\mathbf{x}))=0
\end{equation}
for any $\mathbf{x}=(x_1,\ldots,x_n)\in\bA_{\bbF}^n$. In order to find such a polynomial, we use linear algebra. There are $\binom{M+n+1}{n+1}$ coefficients to be determined. Moreover, since each $P_i$ has degree at most $D$, the left hand side of Equation \eqref{polyrel} is a polynomial of degree at most $DM$ in $\mathbf{x}$, whence imposing \eqref{polyrel} amounts to requiring that these $\binom{M+n+1}{n+1}$ coefficients satisfy a system of at most $\binom{DM+n}{n}$ equations. Choosing $M$ depending only on $D$ and $n$ such that $$\binom{M+n+1}{n+1}>\binom{DM+n}{n}\text{,}$$ which is clearly possible, this system of linear equations is guaranteed to have a non-trivial solution. Since $P_1,\ldots,P_n$ are algebraically independent, the coefficient of some term involving $x_k$ is non-zero.

Now let $\mathbf{y}=(y_1,\ldots,y_n)\in\bA_{\bbF}^n$. If for all $1\leq k\leq n$ the univariate polynomial $Q_k(x,y_1,\ldots,y_n)$ is not constant, then \eqref{polyrel} implies that, over all $\mathbf{x}=(x_1,\ldots,x_n)$ such that $\varphi(\mathbf{x})=\mathbf{y}$, there are only finitely many possible values of $x_k$, and hence the set of such $\mathbf{x}$ is finite. Therefore, if the fibre of $\varphi$ over $\mathbf{y}$ is not zero-dimensional, then $\mathbf{y}$ lies on the vanishing locus of the leading coefficient of $Q_k(x_k,\mathbf{y})$, regarded as a polynomial in $\mathbf{y}$. This implies that such $\mathbf{y}$ lie in the union of $n$ hypersurfaces in $\bA_{\bbF}^n$, each having degree at most $M$. Taking $W$ to be this union, of degree at most $nM=:C_{D,n}$, this proves the lemma.
\end{proof}
\begin{proof}[Proof of Theorem \ref{maingen}]

Assume $P$ is nice, and without loss of generality write
$$P(x_1,\ldots,x_{d+1})=ax_{d+1}^d+\sum_{j=1}^d P_j(x_1,\ldots,x_{d})
x_{d+1}^{d-j}$$
with $P_1,\ldots,P_d$ algebraically independent. Define $\varphi:\bbF_q^d\to\bbF_q^d$ by
$$\varphi(x_1,\ldots,x_d)=(P_1(x_1,\ldots,x_d),\ldots,P_d(x_1,\ldots,x_d))\text{.}$$
For each $\mathbf{a}=(a_1,\ldots,a_d)\in\bbF_q^d$, let
$$\fq_\mathbf{a} = 
\{ (s,t) \in \bbF_q^2:
t=a_1s^{d-1} +\cdots+ a_{d-1}s+a_d \} \text{.}$$
We assume as we may that $q \ge d$, so that $\fq_\ba$ uniquely determines $\ba$.

We claim that, for an appropriate choice of $C$ in the statement of Theorem~\ref{maingen}, there exists $c>0$ such that if
$$\mathcal{Q}_{X_1,\ldots, X_d}=\{\fq_{\varphi(\bx)}:\bx\in X_1\times\cdots\times X_d\}$$
then
$$|\mathcal{Q}_{X_1,\ldots, X_d}| \geq c|X_1|\cdots|X_d|
\text{.}$$
To this end, we argue as follows. Identify $\bbF_q^d$ with affine space $\bA_{\bbF_q}^d$. The condition that $P_1,\ldots,P_d$ are algebraically independent implies that the map $\varphi$ is dominant, i.e. its image is not contained in a proper subvariety of $\bA_{\bbF_q}^d$. Therefore, if we denote by $k=O_d(1)$ the degree of $\varphi$, there exists a proper subvariety $X\subseteq\bA_{\bbF_q}^d$, of degree $O_d(1)$ by Lemma \ref{bounddeg}, such that the fibre of $\varphi$ over $\mathbf{a}$ is finite whenever $\mathbf{a}\notin X$. Since $\varphi$ is at most $k$-to-$1$ outside $\varphi^{-1}(X)$, it follows that
\begin{align*}
|\mathcal{Q}_{X_1,\ldots, X_d}|&\geq|\varphi((\bA_{\bbF_q}^d\setminus\varphi^{-1}(X))(\bbF_q)\cap (X_1\times\cdots\times X_d))|\\
&\geq\frac{1}{k}|(\bA_{\bbF_q}^d\setminus\varphi^{-1}(X))(\bbF_q)\cap (X_1\times\cdots\times X_d)|\\
&\geq\frac{1}{k}(|X_1|\cdots|X_d|-|\varphi^{-1}(X)(\bbF_q)|)\\
&=\frac{1}{k}(|X_1|\cdots|X_d| + O_d(q^{d-1}))\text{.}
\end{align*}
In the final line we used that $\varphi^{-1}(X)$, being a proper subvariety of $\bA_{\bbF_q}^d$, has dimension at most $d-1$, and its degree is bounded in terms of the degrees of $X$ and $\varphi$. It is now clear that, for sufficiently large $C$, if $|X_i|\geq Cq^{d/(d+1)}$ for all $i$, then the above is at least $c|X_1|\cdots|X_d|$ for a suitably small constant $c > 0$.

Next, let $W=\bbF_q\setminus P(X_1,\ldots,X_{d+1})$, and define $\psi:\bbF_q^2\to\bbF_q^2$ by 
$$\psi(x,w)=(x,w-ax^d)\text{.}$$
Note that $\psi$ is bijective. Moreover, from the definition of $W$ it follows that, for any $w\in W$, $x\in X_{d+1}$ and $\bx=(x_1,\ldots,x_{d})\in X_1\times\cdots\times X_{d}$, the point $\psi(x,w)$ does not lie in $\fq_{\varphi(\bx)}$. 
For $\cP_{X_{d+1},W}=\psi(X_{d+1}\times W)$, we thus have 
$$\{ (\fp, \fq) \in \cP_{X_{d+1},W} \times \cQ_{X_1,\ldots,X_d}: \fp \in \fq \}=\emptyset$$
and, by the injectivity of $\psi$, \begin{equation}
\label{xw}
|\cP_{X_{d+1},W}|=|X_{d+1}|\cdot |W|\text{.}
\end{equation} 

It follows from Theorem \ref{higher} that
$$\frac{|\cP_{X_{d+1},W}|\cdot|\cQ_{X_1,\ldots,X_d}|}{q}\leq q^{(d-1)/2}\sqrt{|\cP_{X_{d+1},W}|\cdot|\cQ_{X_1,\ldots,X_d}|}\text{,}$$
which rearranges to
$$|\cP_{X_{d+1},W}|\cdot|\cQ_{X_1,\ldots,X_d}|\leq q^{d+1}\text{.}$$
Using \eqref{xw} and the earlier conclusion that $|\cQ_{X_1,\ldots,X_d}|\gg|X_1|\cdots|X_d|$, this yields
$$|W|\cdot|X_1|\cdots|X_{d+1}|\ll q^{d+1}$$
which, since $|W|=q-|P(X_1,\ldots,X_{d+1})|$, simplifies to the desired result.
\end{proof}

We now prove Lemma \ref{nicequad}, establishing the more symmetric and satisfactory description of the `nice' condition in the case of quadratic polynomials. The proof is a long computation and, in order to make it easier to parse, we prove first a version for homogeneous polynomials and then deduce the general version from it.

\begin{lemma}
\label{homcase}
Let $q$ be an odd prime power and let $Q\in\bbF_q[x,y,z]$ be a ternary quadratic form. Then $Q$ is nice if and only if the following conditions hold:
\begin{enumerate}[label=(\roman*)]
\item $Q$ is non-diagonal;
\item $Q$ is \emph{genuinely ternary}, in the sense that each variable appears in a monomial having non-zero coefficient in $Q$;
\item $Q$ is not a square in $\overline{\bbF_q}[x,y,z]$.
\end{enumerate}
\end{lemma}

\begin{proof}
It is straightforward to show that if $Q$ does not satisfy (i), (ii) and (iii) then $Q$ is not nice. We will now show, conversely, that if $Q$ is not nice then $Q$ fails (i), (ii) or (iii).
We write
$$Q(x,y,z)=ax^2+by^2+cz^2+dyz+exz+fxy$$
and argue according to how many of $d,e,f$ are non-zero. If $d=e=f=0$, then $Q$ is diagonal and violates (i).

If exactly one of $d,e,f$ is non-zero, assume without loss of generality that $f \ne 0$. Then the polynomials $L(y,z)=fy$ and $F(y,z)=by^2+cz^2$ are algebraically dependent. It follows that $y$ divides $by^2+cz^2$, so $c=0$. We obtain $$Q(x,y,z)=ax^2+by^2+fxy\text{,}$$ and hence $Q$ fails Condition (ii).

If exactly two of $d,e,f$ are non-zero, assume without loss of generality that $e=0$. Then the polynomials $fy$ and $by^2+cz^2+dyz$ are linearly dependent, and we conclude that $c=0$ as before. Hence 
\[
y, \qquad by^2+dyz=y(by+dz)
\]
are algebraically dependent. We can argue from the definition of algebraic dependence that this is impossible, so this case cannot occur.

Finally, suppose $def \ne 0$. The polynomials $ez+fy$ and $by^2+cz^2+dyz$ are algebraically dependent, which implies that the former divides the latter. This in turn implies that the latter vanishes at $(y,z)=(-e,f)$, i.e.
$$be^2+cf^2=def\text{.}$$
Similarly, we also obtain
$$ad^2+be^2=def, \qquad
ad^2+cf^2=def\text{.}$$
The three equations together imply that $ad^2=be^2=cf^2=def/2$, which yields
$$a=\frac{ef}{2d}\text{,}\quad b=\frac{df}{2e}\text{,}\quad c=\frac{de}{2f}\text{.}$$
Now
$$Q(x,y,z)=\frac{(efx+dfy+dez)^2}{2def}\text{,}$$
so $Q$ fails Hypothesis (iii).
\end{proof}

\begin{lemma}
\label{algdepcri}
Let $L\in\bbF_q[x,y]$ be a non-zero linear form, and let $Q\in\bbF_q[x,y]$ be a quadratic polynomial such that $Q(0,0)=0$. If $L$ and $Q$ are algebraically dependent, then there exist $a,b\in\bbF_q$ such that 
\[
Q=aL^2+bL.
\]
\end{lemma}

\begin{proof}
If $P$ is a polynomial such that $P(L,Q)=0$, then by evaluating at $(0,0)$ it follows that $P$ has constant coefficient $0$, and hence it is clear that $L$, being irreducible, divides $Q$. Writing $Q=LL'$, the linear polynomials $L$ and $L'$ are algebraically dependent. If $L$ and the homogeneous linear part of $L'$ are linearly independent, then $L$ and $L'$ generate $\bbF_q[x,y]$, contradicting algebraic dependence. Therefore, since $L\neq0$, the homogeneous linear part of $L'$ is a multiple of $L$, implying that $L'=aL+b$ for some $a,b\in\bbF_q$, and the result follows.
\end{proof}

\begin{proof}[Proof of Lemma \ref{nicequad}]
Again, we will focus on proving that, if $Q$ is not nice, then $Q$ is of the form given in Lemma \ref{nicequad}. Suppose then that $Q$ is not nice. Note that neither the assumption nor the conclusion are affected by changing the constant term, so we may assume that $Q(0,0,0)=0$. Write 
\begin{align*}
&Q(x,y,z)\\&=a_1x^2+a_2y^2+a_3z^2+b_1yz+b_2xz+b_3xy+c_1x+c_2y+c_3z\text{.}
\end{align*}
If all of $b_1,b_2,b_3$ equal $0$ we are done, so we exclude that case henceforth. Assume first that only one of $b_1,b_2,b_3$ is non-zero; without loss of generality assume that it is $b_3$. The assumption that $Q$ is not nice implies that $b_3y+c_1$ and $a_2y^2+a_3z^2+c_2y+c_3z$ are algebraically dependent, which implies that $a_3=c_3=0$. But then $Q$ is not genuinely ternary.

Assume now that at most one of $b_1,b_2,b_3$ equals $0$. The assumption that $Q$ is not nice implies that
$$b_2z+b_3y+c_1,\qquad a_2y^2+a_3z^2+b_1yz+c_2y+c_3z$$
are algebraically dependent. It now follows from Lemma \ref{algdepcri} that 
\[
b_2c_2=b_3c_3\text{,}
\]
and moreover $a_2y^2+a_3z^2+b_1yz$ is a constant multiple of $(b_2z+b_3y)^2$.

If $b_3=0$, say, then the above implies that $a_2=b_1=0$, which contradicts the assumption that $b_1$ and $b_2$ are non-zero. So $b_1,b_2,b_3$ are all non-zero, and the observation in the previous paragraph, together with its two natural analogues obtained by permuting the indices, implies that $b_1c_1=b_2c_2=b_3c_3=s$ for some $s\in\bbF_q$, and that the homogeneous quadratic part of $Q$ is itself not nice. Therefore, by Lemma \ref{homcase}, there exist $d_1,d_2,d_3,e\in\bbF_q \setminus \{0\}$ such that
\begin{align*}
&e(d_1x+d_2y+d_3z)^2\\
&=a_1x^2+a_2y^2+a_3z^2+b_1yz+b_2xz+b_3xy\text{.}
\end{align*}
This implies that $a_i=ed_i^2$ and $b_i=2d_{i+1}d_{i+2}$ for $i\in\{1,2,3\}$, where $d_4 = d_1$ and $d_5 = d_2$. This, in turn, implies that $c_i=s/(2d_{i+1}d_{i+2})$ for $i\in\{1,2,3\}$, whence
\begin{align*}Q(x,y,z)&=e(d_1x+d_2y+d_3z)^2+\frac{s}{2d_2d_3}x+\frac{s}{2d_1d_3}y+\frac{s}{2d_1d_2}z\\
&=e\left(d_1x+d_2y+d_3z+\frac{s}{4ed_1d_2d_3}\right)^2-\frac{s^2}{16ed_1^2d_2^2d_3^2}\text{,}
\end{align*}
This shows that $Q$ has the desired form.
\end{proof}

\section{Ternary polynomials of higher degree}
\label{TernaryHigher}

In this section we prove Theorem \ref{conc}, showcasing the adaptability of our methods for certain polynomials of degree $d$ in as few as three variables.

\begin{proof}[Proof of Theorem \ref{conc}]
Repeating the argument in the proof of Theorem \ref{maingen} we see that, for a suitable choice of $C$, there exists a constant $c > 0$ such that the set $\cL_{Y,Z}$ of lines in $\bbF_q^2$ of the form
$$\{(s,t)\in\bbF_q^2:t=F(y,z)s+G(y,z)\}$$
with $(y,z)\in Y\times Z$ has size at least $c|Y|\cdot|Z|$.

As in the proof of Theorem \ref{maingen}, define $W=\bbF_q\setminus P(X,Y,Z)$, and let $\psi:\bbF_q^2\to\bbF_q^2$ be given by $\psi(x,w)=(x,w-ax^d)$. The set $\cP_{X,W}=\psi(X\times W)$ has size $|X|\cdot |W|$, and
$$\{(\fp,\fl)\in\cP_{X,W}\times\cL_{Y,Z}:\fp\in\fl\}=\emptyset\text{.}$$
Then Theorem \ref{VinhThm} implies that
$$|\cP_{X,W}|\cdot|\cL_{Y,Z}|\leq q^3\text{,}$$
whence
\[
|W|\ll \frac{q^3}{|X|\cdot|Y|\cdot|Z|}.
\]
The result now follows using the definition of $W$.

\end{proof}

\providecommand{\bysame}{\leavevmode\hbox to3em{\hrulefill}\thinspace}

\end{document}